\documentclass[reqno,12pt]{amsart}

\usepackage[margin=1.2in]{geometry}
\usepackage{amsmath,amssymb,amsthm}
\usepackage{dsfont}
\usepackage{graphicx}
\usepackage{microtype}

\usepackage[dvipsnames]{xcolor}
\usepackage{hyperref}
\hypersetup{
    colorlinks=true,
    linkcolor=cyan!80!black,
    citecolor=MidnightBlue,
    urlcolor=magenta,
}

\theoremstyle{plain}
\newtheorem{theorem}{Theorem}
\newtheorem{corollary}{Corollary}
\newtheorem{proposition}{Proposition}
\newtheorem{lemma}{Lemma}
\newtheorem{remark}{Remark}

\theoremstyle{definition}

\newtheorem{example}{Example}
\newtheorem{definition}{Definition}

\newcommand{\bbP}{\mathbb{P}}
\newcommand{\bbR}{\mathbb{R}}
\newcommand{\bbZ}{\mathbb{Z}}
\newcommand{\bbE}{\mathbb{E}}

\newcommand{\bbN}{\mathbb{N}}

\newcommand{\bV}{\mathbf{V}}

\newcommand{\bZ}{\mathbf{Z}}
\newcommand{\bv}{\mathbf{v}}
\newcommand{\br}{\mathbf{r}}
\newcommand{\bS}{\mathbf{S}}

\newcommand{\bone}{\mathbf{1}}

\newcommand{\convas}{\xrightarrow{\text{a.s.}}}

\newcommand{\cP}{\mathcal{P}}

\newcommand{\op}{\mathrm{op}}

\newcommand{\spec}{\mathrm{spec}}
\DeclareMathOperator{\Var}{Var}
\DeclareMathOperator{\Tr}{Tr}

\DeclareMathOperator{\Cov}{Cov}
\newcommand{\FK}{{F{\"u}redi--Koml{\'o}s}}
\newcommand{\wt}{\mathsf{wt}}
\newcommand{\supp}{\mathsf{supp}}

\let\tilde\widetilde

\title[Edge spectra of Gaussian matrices with correleated entries]{Edge spectra of Gaussian random symmetric matrices with correlated entries}

\author[D. Banerjee]{Debapratim Banerjee}
\address{
    Department of Mathematics \\
    Ashoka University\\
    Plot no 2, Rajiv Gandhi Education City, Sonipat 131029\\
    Haryana, India.
}
\email{debapratim.banerjee@ashoka.edu.in}

\author[S. S. Mukherjee]{Soumendu Sundar Mukherjee}
\address{
    Statistics and Mathematics Unit \\
    Indian Statistical Institute \\
    203 B.T. Road, Kolkata 700108 \\
    West Bengal, India.
}
\email{ssmukherjee@isical.ac.in}

\author[D. Pal]{Dipranjan Pal}
\address{
    Statistics and Mathematics Unit \\
    Indian Statistical Institute \\
    203 B.T. Road, Kolkata 700108 \\
    West Bengal, India.
}
\email{dipranjan\_r@isical.ac.in}

\begin{document}

\begin{abstract}
In this article, we study the largest eigenvalue of a Gaussian random symmetric matrix $X_n$, with zero-mean, unit variance entries satisfying the condition $\sup_{(i, j) \ne (i', j')}|\mathbb{E}[X_{ij} X_{i'j'}]| = O(n^{-(1 + \varepsilon)})$, where $\varepsilon > 0$. It follows from \cite{catalano2024random} that the empirical spectral distribution of $n^{-1/2} X_n$ converges weakly almost surely to the standard semi-circle law. Using a F\"{u}redi-Koml\'{o}s-type high moment analysis, we show that the largest eigenvalue $\lambda_1(n^{-1/2} X_n)$ of $n^{-1/2} X_n$ converges almost surely to $2$. This result is essentially optimal in the sense that one cannot take $\varepsilon = 0$ and still obtain an almost sure limit of $2$. A simple application of the remarkably general universality results of \cite{brailovskaya2024universality} shows the universality of this convergence in a broad class of random matrices arising as random linear combinations of deterministic matrices. We also derive Gaussian fluctuation results for the largest eigenvalue in the case where the entries have a common non-zero mean. Let $Y_n = X_n + \frac{\lambda}{\sqrt{n}}\mathbf{1} \mathbf{1}^\top$. When $\varepsilon \ge 1$ and $\lambda \gg n^{1/4}$, we show that
\[
    n^{1/2}\bigg(\lambda_1(n^{-1/2} Y_n) - \lambda - \frac{1}{\lambda}\bigg) \xrightarrow{d} \sqrt{2} Z,
\]
where $Z$ is a standard Gaussian. On the other hand, when $0 < \varepsilon < 1$, we have $\Var(\frac{1}{n}\sum_{i, j}X_{ij}) = O(n^{1 - \varepsilon})$. Assuming that $\Var(\frac{1}{n}\sum_{i, j} X_{ij}) = \sigma^2 n^{1 - \varepsilon} (1 + o(1))$, if $\lambda \gg n^{\varepsilon/4}$, then we have
\[
    n^{\varepsilon/2}\bigg(\lambda_1(n^{-1/2} Y_n) - \lambda - \frac{1}{\lambda}\bigg) \xrightarrow{d} \sigma Z.
\]
While the ranges of $\lambda$ in these fluctuation results are certainly not optimal, a striking aspect is that different scalings are required in the two regimes $0 < \varepsilon < 1$ and $\varepsilon \ge 1$.
\end{abstract}

\maketitle

\begin{center}
\end{center}

\thispagestyle{empty}

\section{Introduction}\label{sec:intro}

Traditionally random matrix theory has considered matrix models with independent entries. Spectacular progress has been made on these independent models over the last two decades resulting in the resolution of the so-called Wigner-Dyson-Mehta conjecture \cite{mehta2004random,erdHos2010bulk, MR2661171, MR2851058, erdHos2012local,erdHos2012bulk, erdHos2012rigidity, ajanki2017universality}.

There has also been a steady stream of works on ensembles of random matrices where the entries are correlated. An incomplete list of works include \cite{boutet1996limiting, hachem2005empirical, schenker2005semicircle, pastur2011eigenvalue, chakrabarty2013limiting, gotze2015limit, hochstattler2016semicircle, chakrabarty2016random, ajanki2016local, che2017universality, erdHos2019random, alt2020correlated, goulart2022random, au2023spectral, catalano2024random, bonnin2024universality, mukherjee2024spectra, chakrabarty2024largest}.

In \cite{che2017universality} bulk universality was obtained under the assumption that the entries are $k$-dependent for some fixed $k$. A much more general model was considered in \cite{erdHos2019random}, where the authors imposed appropriate decay rates on multivariate cumulants (see Assumption (CD) in \cite{erdHos2019random}). Under these relaxed assumptions the authors proved bulk universality (see \cite[Corollary 2.6]{erdHos2019random}).

For edge rigidity and edge universality, one might look at \cite{alt2020correlated} and \cite{adhikari2019edge}. These works use the Green's function approach which is much successful in the independent setting. However, as pointed out in \cite{adhikari2019edge}, the Green's function approach becomes significantly more difficult when one has more and more correlations among the entries. One needs appropriate correlation decay hypotheses to execute this approach. In particular, for matrices $X$ with jointly Gaussian entries \cite{adhikari2019edge} assume the following correlation decay:
\begin{equation}\label{eq:corr-cond-ac19}
    |\Cov(X_{ij},X_{kl})| \le C \max\bigg\{\frac{1}{(|i-k|+|j-l|+1)^{d}},\frac{1}{(|i-l|+|j-k|+1)^{d}} \bigg\}
\end{equation}
for $d>2$ and some constant $C > 0$. Further, in \cite{alt2020correlated} the authors allow non-Gaussian entries and a similar power law correlation decay with exponent $d > 12$ (see Assumption (CD) in \cite{alt2020correlated}).

In this paper, we study the edge of the spectrum of correlated Gaussian matrices where the correlations decay like $O\left(\frac{1}{n^{1+\varepsilon}}\right)$ for a fixed $\varepsilon >0$. This setting neither implies nor is implied by \eqref{eq:corr-cond-ac19}. Indeed, one can have very high correlation among nearby entries as per \eqref{eq:corr-cond-ac19}; however, when the entries are far away, \eqref{eq:corr-cond-ac19} stipulates a much faster correlation decay. Thus when the entries are at distance $\Omega(n)$, \cite{adhikari2019edge}  assumes that the correlation-decay is of order $\frac{1}{n^{d}}$ for $d > 2$, which is much faster than our assumed decay rate of $n^{-(1 + \varepsilon)}$. As the authors of \cite{adhikari2019edge} point out, it is believed that $d > 2$ is the optimal regime where one might expect to prove universality estimates in these types of models.

It is well known that moment based techniques, despite their apparent crudeness, are remarkably robust. In fact, using the moment method, it was shown in the recent work \cite{catalano2024random} that when the correlations are all $\le \frac{1}{n}$, the empirical spectral distribution converges weakly almost surely to the standard semi-circle law (see Corollary 2.7 in \cite{catalano2024random}). This of course includes our setting where the correlations are uniformly $O(n^{-(1 + \varepsilon)})$. A natural question therefore is if the largest eigenvalue converges to $2$, the right end-point of the support of the standard semi-circle law. Employing the \emph{method of high moments} of \cite{furedi1981eigenvalues}, we show that this is indeed the case when $\varepsilon > 0$ (see Theorem~\ref{thm:largest_ev}). The criterion $\varepsilon > 0$ is essentially optimal as we demonstrate that edge rigidity does not necessarily hold when $\varepsilon=0$. (see Remarks~\ref{rem:non-universality} and \ref{rem:hypergraph}). We also establish Gaussian fluctuations of the largest eigenvalue in the spiked model ensuing from all entries having a sufficiently large common positive mean (see Theorem~\ref{thm:fluc}). Notably different scalings are required in the regimes $0 < \varepsilon < 1$ and $\varepsilon \ge 1$.

Incidentally, \cite{reker2022operator} carried out a moment method analysis for matrices with general entries and correlation decay of the form \eqref{eq:corr-cond-ac19} with exponent $d > 2$ (assuming further decay conditions on multivariate cumulants to deal with non-Gaussianity) to prove that the operator norm (hence the largest eigenvalue) is stochastically bounded by $1$. Closer to our work is \cite{fleermann2023large}, where the authors consider Gaussian Wishart matrices with a correlation decay similar to ours and prove a result analogous to Theorem \ref{thm:largest_ev} in that setting. However, they only study the centered case. We also note that the recent work \cite{chakrabarty2024largest} establishes Gaussian fluctuations for the largest eigenvalue for Gaussian matrices with a common positive mean, where the entries are obtained by symmetrising a stationary Gaussian field indexed by $\bbZ^2$ with an absolutely summable covariance kernel. We note however that our covariance decay condition is much weaker --- in fact, the $n \times n$ partial sums of our covariances can be as large as $O(n^{1 - \varepsilon})$. Further, the common positive mean considered in \cite{chakrabarty2024largest} is $\Theta(1)$, whereas we allow much smaller perturbations (e.g., when $\varepsilon > 1$, we can handle a common mean of order $\gg n^{-1/4}$).

\subsection{The model and our main result on the largest eigenvalue.}
Let $\left(X_{ij}\right)_{1 \le i\le j \le n}$ be a centered multivariate Gaussian vector of dimension $n(n+1)/2$ with $\Var(X_{ij}) = 1$ for all $i \le j$ and
\begin{align}\label{eq:correlation}
   \sup_{(i, j) \ne (i', j')} \big|\bbE[X_{ij} X_{i'j'}]\big| &= O\bigg(\frac{1}{n^{1 + \varepsilon}}\bigg),
\end{align}    
where $\varepsilon > 0$ is a fixed constant. Given this multivariate Gaussian vector, we consider the (symmetric) matrix $X_{n}$ with $X_{n}(i,j) = X_{ij}$ and $X_{n}(i,j)= X_{n}(j,i)$ for $i\le j$. 

Our main result is the following:
\begin{theorem}\label{thm:largest_ev}
    Let $X_{n}$ be the symmetric Gaussian random matrix descirbed above. Then $\lambda_1(n^{-1/2}X_n) \to 2$ almost surely.
\end{theorem}

\begin{remark}
    Although, for the sake of simplicity, we have assumed that $\Var(X_{ij}) = 1$ for all $i, j$, it is not difficult to see that our results continue to hold if
    \begin{equation}\label{eq:variance}
        \sup_{1 \le i \le j \le n} |\Var(X_{ij}) - 1| = O\left(\frac{1}{(\log n)^2 }\right).
    \end{equation}
    This will be the case in several examples later.
\end{remark}
\begin{remark}
    By following the proof of Theorem \ref{thm:largest_ev} carefully, one can see that \eqref{eq:correlation} can be strengthened to 
    \begin{align*}\label{eq:correlationstrong}
   \sup_{(i, j) \ne (i', j')} \big|\bbE[X_{ij} X_{i'j'}]\big| &= O\bigg(\frac{1}{n (\log n)^\tau}\bigg),
\end{align*}
for any $\tau >4$.
\end{remark}
\begin{figure}[!t]
    \centering
    \includegraphics[width = 0.9\textwidth]{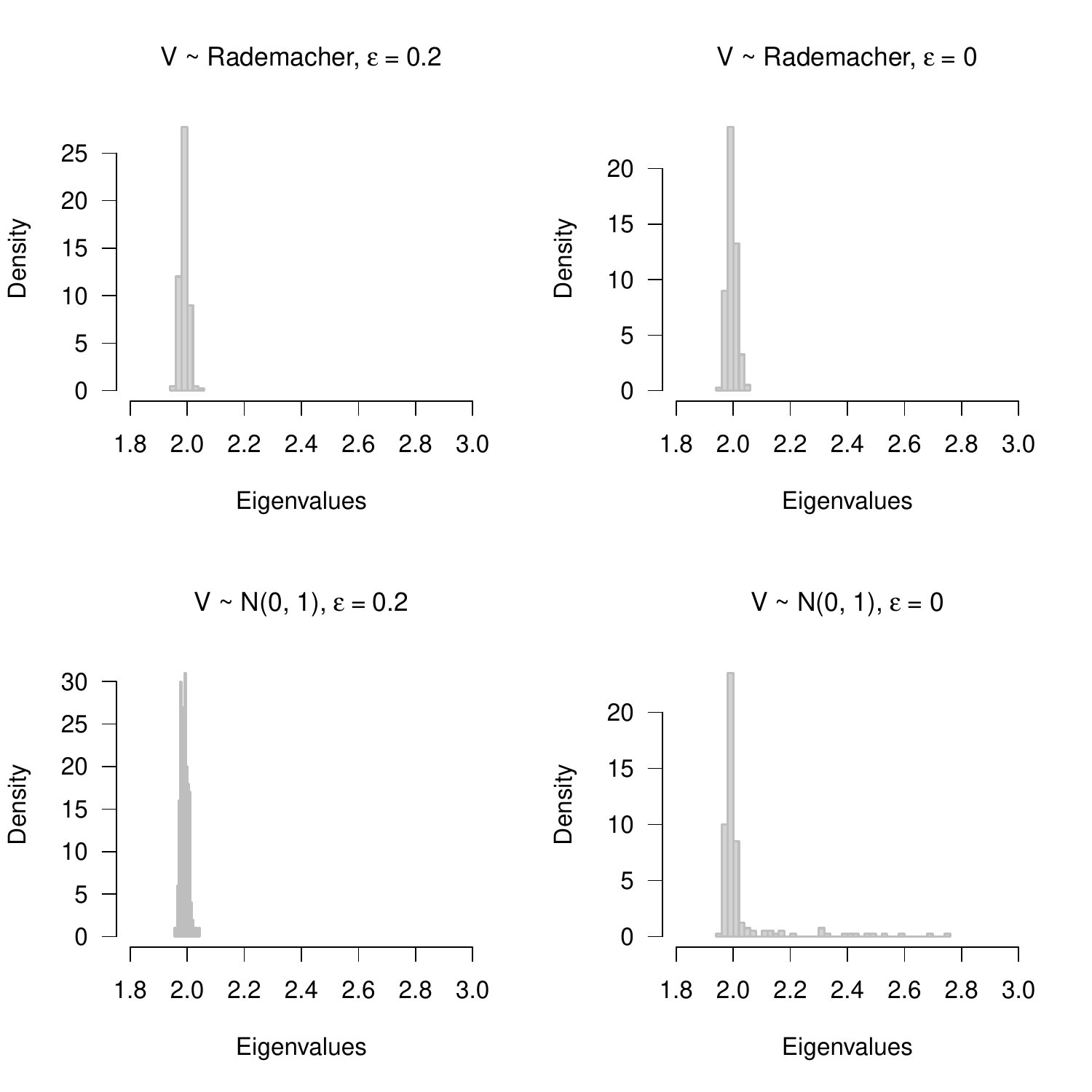}
    \caption{Histograms of the largest eigenvalue of $n^{-1/2} X_n$ from the model described in \eqref{eq:test_model} with $\alpha_n^2 = n^{-(1 + \varepsilon)}$ for $n = 1000$ based on $200$ Monte Carlo simulations. In the fourth setting, we empirically observed that in about $84\%$ of the Monte Carlo runs the largest eigenvalue fell within the range $[1.968, 2.032]$.}
    \label{fig:edge}
\end{figure}

\begin{remark}\label{rem:non-universality}
We note here that the correlation condition in \eqref{eq:correlation} cannot be dropped to $O(1/n)$. To see this, consider the following test model:
\begin{equation}\label{eq:test_model}
    X_n = W_n + \alpha_n V \bone \bone^\top,
\end{equation}
where $n^{-1/2} W_n$ is a GOE random matrix, and $V$ is an independent random variable with zero mean and unit variance. Note that for $\{i, j\} \ne \{i', j'\}$,
\[
    \Cov(X_{ij}, X_{i'j'}) = \Cov(W_{ij} + \alpha_n V, W_{i'j'} + \alpha_n V) = \alpha_n^2.
\]
Also, $\Var(X_{ij}) = 1 + \alpha_n^2$. If we take $\alpha_n = \frac{1}{\sqrt{n}}$ and $V$ is a Rademacher random variable (i.e. a random sign), then from the Baik-Ben Arous-P\'{e}ch\'{e} phase transition for spiked Wigner models \cite{BBPtrransition2005, Mireillelargesteigen2009, capitaine2009largest, benaych2011},
we see that the largest eigenvalue of $n^{-1/2} X_n$ converges almost surely to $2$.

On the other hand, if $V$ is itself a standard Gaussian, then with probability only $\bbP(V \le 1) \approx 0.84$, the largest eigenvalue of $n^{-1/2} X_n$ converges to $2$.
\end{remark}

In Figure~\ref{fig:edge}, we show the histograms of the largest eigenvalues of model described in \eqref{eq:test_model} in several different settings. This empirically demonstrates the non-universality mentioned in Remark~\ref{rem:non-universality}.

\begin{remark}\label{rem:hypergraph}
    Another situation where the correlations are all $O(1/n)$  but a limit other than $2$ emerges is the case of adjacency matrices of Erd\H{o}s-R\'{e}nyi $r$-uniform hypergraphs for fixed $r$ \cite{mukherjee2024spectra}. In fact, in this model only $\Theta(n^2)$ of the $\Theta(n^4)$ correlations are $\Theta(1/n)$ and the rest are $\Theta(1/n^2)$. It was shown in \cite{mukherjee2024spectra} that when $r \ge 4$, then the largest eigenvalue converges almost surely to $\sqrt{r - 2} + \frac{1}{\sqrt{r - 2}}$.
\end{remark}

\begin{remark}
In \cite{mukherjee2024spectra}, the following Gaussian random matrix was considered:
\[
    X_n = \alpha_n U \bone \bone^\top + \beta_n (\bone \bV^\top + \bV \bone^\top) + \theta_n Z_n
\]
where $U$, $\bV = (V_i)_{1 \le i \le n}$ are i.i.d. standard Gaussian random variables and $n^{-1/2} Z_n$ is an independent GOE random matrix. Note that for $i \neq j$ and $i' \neq j'$,
\begin{align*}
   \Cov(X_{n,ij}, X_{n,i'j'}) &= \begin{cases}
        \alpha_n^2 & \text{if } |\{i, j\} \cap \{i', j'\}| = 0, \\
        \alpha_n^2 + \beta_n^2 & \text{if } |\{i, j\} \cap \{i', j'\}| = 1, \\
        \alpha_n^2 + 2 \beta_n^2 + \theta_n^2 & \text{if } |\{i, j\} \cap \{i', j'\}| = 2.
   \end{cases}
\end{align*}
In addition, 
\[
    \Var(X_{n,ii}) = \alpha_n^2 + 4 \beta_n^2 + \theta_n^2.
\]
Let $\gamma < 1 + \varepsilon$. If we set $\alpha_n = n^{-(1+\varepsilon)/2}$, $\beta_n = \sqrt{n^{-\gamma} - n^{-(1 + \varepsilon)}}$, and $\theta_n = \sqrt{1 - \alpha_n^2 - 2\beta_n^2} = \sqrt{1 - 2 n^{-\gamma} + n^{-(1 + \varepsilon)}}$, then
\[
   \Cov(X_{n, ij}, X_{n, i^{'}j^{'}}) = \begin{cases}
       \frac{1}{n^{1 + \varepsilon}} & \text{if} \ \ \big|\{i, j\} \cap \{i^{'}, j^{'} \}\big| = 0, \\
       \frac{1}{n^{\gamma}} & \text{if} \ \ \big|\{i, j\} \cap \{i^{'}, j^{'} \}\big| = 1,
   \end{cases}
\]
$\Var(X_{n, ij}) = 1$ for $i \ne j$ and $\Var(X_{n, ii}) = 1 + O(n^{-\gamma})$. Note here that only $\Theta(n^2)$ of the $\Theta(n^4)$ correlations are of a higher order (namely, $n^{-\gamma}$), the rest being $O(n^{-(1 + \varepsilon)})$. For this model, it can be shown that
if $\gamma \ge 1$, then $\lambda_1(n^{-1/2} X_n)$ converges almost surely to $2$ and if $\gamma < 1$, then $\lambda_1(n^{-(2 - \gamma)/2}X_n) \to 1$ almost surely. This example shows that $\Theta(n^2)$ of the correlations can be increased up to order $n^{-1}$ while preserving the almost sure limit of $2$ for $\lambda_1(n^{-1/2} X_n)$. However, any further increase will lead to a blow-up.
\end{remark}

\begin{figure}[!t]
    \centering
    \includegraphics[width = 0.9\textwidth]{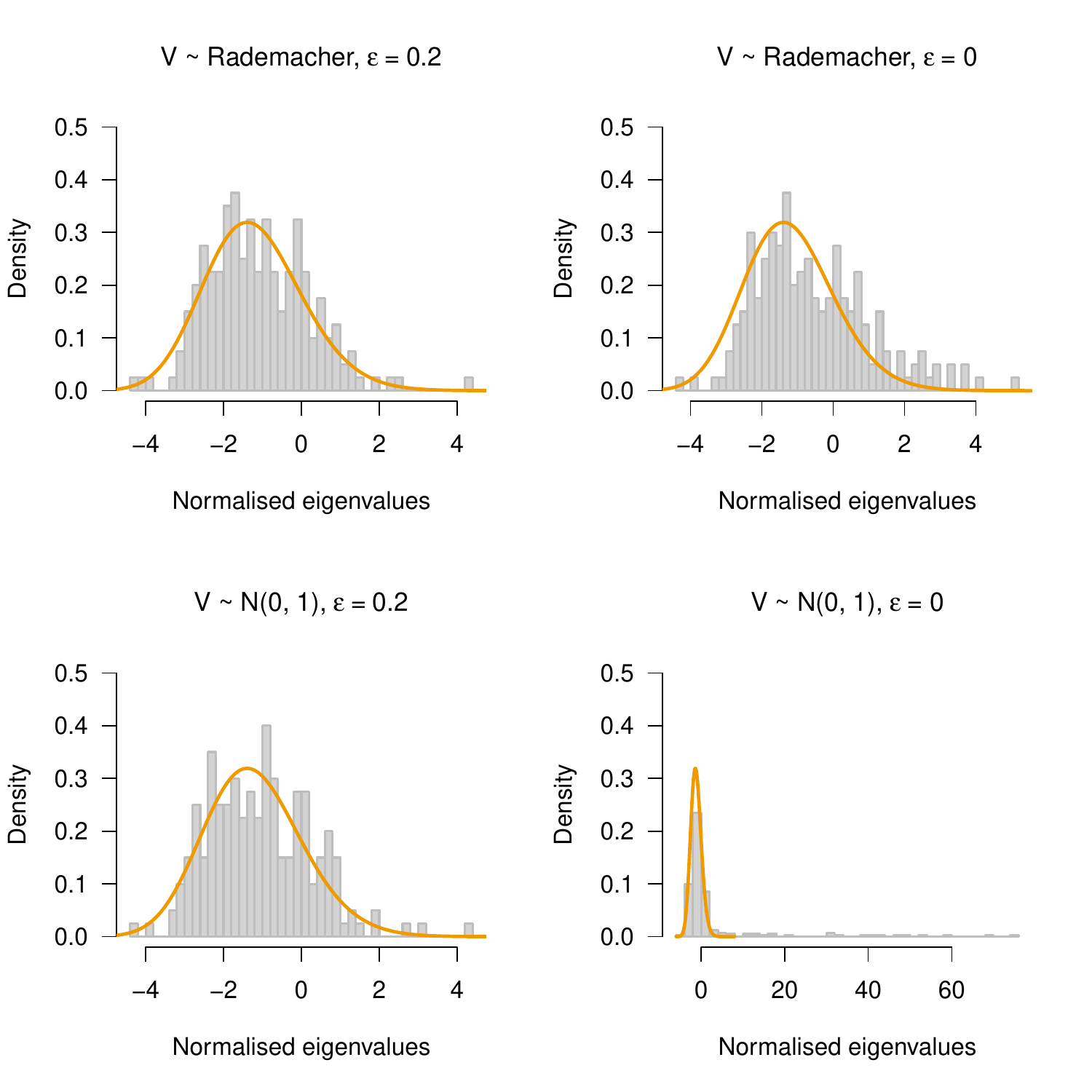}
    \caption{Histograms of $n^{2/3}(\lambda_1(n^{-1/2} X_n) - 2)$ from the model described in \eqref{eq:test_model} with $\alpha_n^2 = n^{-(1 + \varepsilon)}$ for $n = 1000$ based on $200$ simulations. The orange curves depict the density of the GOE Tracy-Widom distribution.}
    \label{fig:edge_fluc}
\end{figure}

Based on simulations shown in Figure~\ref{fig:edge_fluc}, we suspect Tracy-Widom fluctuations for the largest eigenvalue (after centering at $2$ and scaling by $n^{2/3}$) under the same correlation constraints. This will be studied in a future work.

We now present an example of a general class of non-Gaussian matrices obeying the correlation constraint \eqref{eq:correlation}.
\begin{example}
    Let $N$ be a positive integer potentially dependent on $n$. Consider $N$ deterministic matrices $Q_{\ell}, 1 \le \ell \le N$ satisfying the following two conditions:
    \begin{align}\label{ass:Q_entries_1}
        \sup_{1 \le i \le j \le n} \bigg|\sum_{\ell = 1}^N (Q_\ell)^2_{ij} - 1\bigg| &= O\bigg(\frac{1}{(\log n)^2}\bigg); \\ \label{ass:Q_entries_2}
        \sup_{(i, j) \ne (i', j')} \bigg|\sum_{\ell = 1}^N (Q_\ell)_{ij} (Q_{\ell})_{i'j'}\bigg| &= O\bigg(\frac{1}{n^{1 + \varepsilon}}\bigg). 
    \end{align}
    Now consider a random matrix of the form
    \begin{equation}\label{eq:linear_ensemble}
        X_n \equiv X_n(\mathbf{\Psi}) = \sum_{\ell = 1}^N \psi_\ell Q_\ell,
    \end{equation}
    where $\mathbf{\Psi} = (\psi_{\ell})_{1 \le \ell \le N}$ is a vector of independent zero mean unit variance random variables. Note then that for all $i, j$, $\bbE[X_{ij}] = 0$ and
    \begin{align*}
        \bbE[X^2_{n, ij}] = \sum_{\ell = 1}^N \bbE[\psi^2_{\ell}] (Q_\ell)^2_{ij}  + 2\sum_{ \ell \neq \ell'} \bbE[\psi_{\ell} \psi_{\ell'}]  (Q_\ell)_{ij}  (Q_{\ell'})_{ij}  = \sum_{\ell = 1}^N (Q_\ell)^2_{ij}.
    \end{align*}
    Further, for $(i,j) \neq (i', j')$,
    \begin{align*}
        \bigg|\bbE[X_{n, ij} X_{n, i'j'}]\bigg| = \bigg|\bbE\bigg[ \sum_{\ell = 1}^N \psi_\ell  (Q_\ell)_{ij} \sum_{\ell' = 1}^N \psi_{\ell'} (Q_{\ell'})_{i'j'} \bigg]\bigg|
        = \bigg|\sum_{\ell = 1}^N (Q_\ell)_{ij} (Q_{\ell})_{i'j'} \bigg|. 
    \end{align*}
    Thus the ensemble of matrices described in \eqref{eq:linear_ensemble} satisfy the correlation constraints \eqref{eq:correlation} and the variance condition \eqref{eq:variance}.
\end{example}

The recent work \cite{brailovskaya2024universality} considers matrices of the form $Z = Z_0 + \sum_{i = 1}^N Z_i$, where $Z_0$ is a $ n \times n$ deterministic matrix and $Z_1, \ldots, Z_N$ be any independent $n \times n$ self-adjoint centered random matrices. Using these results, we may prove universality of the largest eigenvalue for matrices of the form \eqref{eq:linear_ensemble}. 

    \begin{proposition}\label{prop:univ}
        Let $\mathbf{\Psi}$ be a random vector with independent zero-mean, unit variance and uniformly bounded co-ordinates. Consider the matrix ensemble  $X_n(\mathbf{\Psi})$ described in \eqref{eq:linear_ensemble}. Suppose further that the matrices $Q_{\ell}, 1 \le \ell \le N$, satisfy the condition 
        \begin{align}\label{ass:operator_norm_Q}
            \max_{1 \le \ell \le N} \|Q_{\ell}\|_{\op} = o \bigg( \frac{\sqrt{n}}{(\log n)^2} \bigg).
        \end{align}
        Then $\lambda_1(n^{-1/2}X_n(\mathbf{\Psi})) \convas 2$.
    \end{proposition}

    \begin{remark}
        Proposition~\ref{prop:univ} uses Theorem~2.6 of \cite{brailovskaya2024universality} under the hood, which makes an assumption of uniform boundedness of the operator norms of the matrices $Z_i$. One may instead use Theorem~2.7 of \cite{brailovskaya2024universality} to remove the assumption of uniform boundedness of the co-ordinates of $\mathbf{\Psi}$ in Proposition~\ref{prop:univ}. We omit the details.
    \end{remark}

    \noindent
    \textbf{A randomised construction of the matrices $Q_{\ell}$.}
    We must show that the constraints \eqref{ass:Q_entries_1}, \eqref{ass:Q_entries_2} and \eqref{ass:operator_norm_Q} are sufficiently general, to allow for a large class of matrices. Towards that end, consider the following geometric interpretation: Associate with each pair $(i, j)$ a vector $\bv^{(ij)} := ((Q_1)_{ij}, (Q_2)_{ij}, \ldots, (Q_N)_{ij}) \in \bbR^N$. Note then that in order to have \eqref{ass:Q_entries_1}, we want the vectors $\bv^{(ij)}$ to be approximately of unit norm in the sense that $\sup_{1 \le i \le j \le n} |\|\bv^{ij}\|_2 - 1| = O(\tfrac{1}{(\log n)^2})$. Further, we must have the uniform approximate orthogonality relation
    \[
        \sup_{(i, j) \ne (i', j')} \big|\langle \bv^{(ij)}, \bv^{(i'j')} \rangle\big| = O(n^{-(1 + \varepsilon)}),
    \]
    for the $Q_\ell$'s to satisfy \eqref{ass:Q_entries_2}. The following proposition show that sufficiently long i.i.d. sparse Radamacher random vectors $\bv^{(ij)}$ satisfy these conditions.
    \begin{proposition}\label{prop:rad-construction}
    Let $N \ge 1$ and $p \in (0, 1/2]$. Let $(\eta^{(ij)}_\ell)_{1 \le i \le j \le n, \, 1 \le \ell \le N}$ be i.i.d. sparse \textsf{Rademacher}$(p)$ random variables, i.e. having distribution  $p \delta_{-1} + (1 - 2p) \delta_{0} + p \delta_1$. Construct $N$ symmetric $n \times n$ matrices $Q_\ell, \ell = 1, \ldots, N$, by setting for $1 \le i \le j \le n$,
    \[
        (Q_{\ell})_{ji} = (Q_{\ell})_{ij} = \frac{\eta^{(ij)}_{\ell}}{\sqrt{2Np}}.
    \]
    \begin{enumerate}
        \item [(a)] Suppose that $p \to 0$ and $np = \Omega(1)$. There is a constant $C > 0$ such that if $N \ge C n^{2 + 2\varepsilon} \log n$ then with probability at least $1 - O(n^{-2})$, the matrices $Q_{\ell}$ satisfy \eqref{ass:Q_entries_1} and $\eqref{ass:Q_entries_2}$.
        \item [(b)] There exist constants $C_0 \ge 1, C_1 > 0$ such that if $p \ll \frac{1}{(\log n)^4}$, $np \ge C_0 \log n$ and $N \ge C_1 n^{2 + 2\varepsilon} \log n$, then with probability at least $1 - O(n^{-2})$, the matrices $Q_{\ell}$ satisfy \eqref{ass:operator_norm_Q}.
    \end{enumerate}
    \end{proposition}

\subsection{Fluctuations of the largest eigenvalue when the entries have non-zero mean}
Now suppose the same setting as \eqref{eq:correlation}, but we have a non-zero mean $\mu = \frac{\lambda}{\sqrt{n}}$ (with $\lambda \le D \sqrt{n}$ for some $D > 0$) for each entry, i.e. we now consider the matrix
\begin{equation}\label{eq:non-zero_mean}
    Y_n = X_n + \mu \bone \bone^\top.
\end{equation}
Let $\lambda_1$ denote the largest eigenvalue of $Y$. To find the fluctuations of $\lambda_1$ we follow the approach of \cite{furedi1981eigenvalues}, suitably modifying it along the way to accommodate our correlation structure \eqref{eq:correlation}.

\begin{theorem}\label{thm:fluc}
Consider the matrix $Y_n$ defined in \eqref{eq:non-zero_mean}, where the entries of $X_n$ satisfy the correlation constraint \eqref{eq:correlation}. We have the following representation for its largest eigenvalue:
\[
    \sqrt{n} \bigg[\lambda_1(n^{-1/2} Y_n) - \bigg(\lambda + \frac{1}{\lambda}\bigg)\bigg] = \frac{1}{n}\sum_{i, j} X_{ij} + \frac{\sqrt{n}}{\lambda} \cdot O_P(n^{-\frac{\min\{\varepsilon, 1\}}{2}})) + O_P\bigg(\frac{\sqrt{n}}{\lambda^2}\bigg).
\]
\end{theorem}

\begin{corollary}\label{cor:fluc}
    Consider the matrix $Y_n$ and let $Z$ be a standard Gaussian.
    
    (a) When $\varepsilon \ge 1$ and $\lambda \gg n^{1/4}$,
    \[
        \sqrt{n} \bigg[\lambda_1(n^{-1/2} Y_n) - \bigg(\lambda + \frac{1}{\lambda}\bigg)\bigg] \xrightarrow{d} \sqrt{2}Z.
    \]

    (b) When $0 < \varepsilon < 1$, we have $\Var[\frac{1}{n}\sum_{ij} X_{ij}] = O(n^{1 - \varepsilon})$. Assuming that  $\Var[\frac{1}{n}\sum_{ij} X_{ij}] = \sigma^2 n^{1 - \varepsilon}(1 + o(1))$, if $\lambda \gg n^{\varepsilon / 4}$, then
    \[
        n^{\varepsilon/2} \bigg[\lambda_1(n^{-1/2} Y_n) - \bigg(\lambda + \frac{1}{\lambda}\bigg)\bigg] \xrightarrow{d} \sigma Z.
    \]
\end{corollary}
Noteworthy here is the phenomenon that different scalings are required in the two regimes $\varepsilon \ge 1$ and $0 < \varepsilon < 1$.

The rest of the paper is organised as follows. Section~\ref{sec:prelim} sets up the combinatorial machinery needed to execute the high-moment analysis. In Section~\ref{sec:proofs}, we then give the details of our proofs. Appendix~\ref{sec:misc-proofs} collects the proofs of some auxiliary results.

\section{Preliminaries}\label{sec:prelim}
The proof Theorem~\ref{thm:largest_ev} is based on a combinatorial analysis of traces of high powers of the matrix $X_{n}$ and is motivated by the arguments of F\"{u}redi-Koml\'{o}s.

We have for any $k$,
\begin{equation}\label{eq:traceexpression}
    \Tr[(n^{-1/2}X_{n})^{k}] = \frac{1}{n^{k/2}}\sum_{i_{1}, i_{2}, \ldots, i_{k}} X_{i_{1}i_{2}}\ldots X_{i_{k}, i_{1}}.
\end{equation}  
We shall analyse the contributions from the tuples of indices $(i_{1}, \ldots, i_{k}, i_{1})$ systematically by careful combinatorial arguments. For this, we shall follow the notations and terminologies given in \cite{anderson2010introduction} and \cite{anderson2006clt}.

\subsection{Words, sentences and their equivalence classes}
\label{subsec:word}
\begin{definition}[$\mathcal{S}$ words]
Given a set $\mathcal{S}$, an $\mathcal{S}$ letter $s$ is simply an element of $\mathcal{S}$. An $\mathcal{S}$ word $w$ is a finite sequence of letters $s_1 \cdots s_k$, at least one letter long. An $\mathcal{S}$ word $w$ is \emph{closed} if its first and last letters are the same. In this paper, $\mathcal{S}=\{1,\ldots,n\}$.
\end{definition}
Two $\mathcal{S}$ words $w_1,w_2$ are called \emph{equivalent}, denoted $w_1\sim w_2$, if there is a bijection on $\mathcal{S}$ that maps one into the other. For any word $w = s_1 \cdots s_k$, we use $l(w) = k$ to denote its \emph{length}. We define the \emph{weight} $\wt(w)$ as the number of distinct elements of the set $\{s_1, \ldots, s_k \}$ and the \emph{support} of $w$, denoted by $\supp(w)$, as the set of letters appearing in $w$. With any word $w$, we may associate an undirected graph, with $\wt(w)$ vertices and at most $l(w) - 1$ edges, as follows.
\begin{definition}[Graph associated with a word]\label{def:graphword}
Given a word $w = s_1 \cdots s_k$, we let $G_w = (V_w, E_w)$ be the graph with set of vertices $V_w = \supp(w)$ and (undirected) edges $E_w = \{\{s_i, s_{i+1} \}, i = 1, \ldots, k - 1 \}.$
\end{definition}
The graph $G_w$ is connected since the word $w$ defines a path connecting all the
vertices of $G_w$, which further starts and terminates at the same vertex if the word
is \emph{closed}.  We note that equivalent words generate the same graphs $G_w$ (up to graph isomorphism) and the same passage-counts of the edges.

\begin{definition}[Sentences and corresponding graphs]\label{def:sentence}
A sentence $a = [w_i]_{i=1}^{m} = [[s_{i,j}]_{j = 1}^{l(w_i)}]_{i = 1}^{m}$ is an ordered collection of $m$ words of lengths $l(w_1), \ldots, l(w_m)$, respectively. We define $\supp(a) := \cup_{i = 1}^m \supp(w_i)$ and $\wt(a) := |\supp(a)|$. We set $G_a = (V_a, E_a)$ to be the graph with
\[
    V_a = \supp(a),\quad E_a = \left\{\{s_{i, j}, s_{i, j + 1}\} \mid j = 1, \ldots, l(w_i) - 1; \, i = 1,\ldots, m\right\}.
\]
\end{definition} 
\subsection{The \FK ~ encoding and bounds}
We now introduce the notion of {\FK} sentences (abbrv. FK sentences). The original idea of {\FK} sentences dates back to \cite{furedi1981eigenvalues}. They can be used to bound the number of words of length $k$. Such bounds are particularly important for proving that the largest eigenvalue of a Wigner matrix converges to $2$. They turn out to be useful in our setting as well.

\begin{definition}[FK sentences]
Let $a = [w_{i}]_{i = 1}^{m}$ be a sentence consisting of $m$ words. We say that $a$ is an \emph{FK sentence} if the following conditions hold:
\begin{enumerate}
    \item $G_a$ is a tree;
    \item jointly the words/walks $w_i$, $i = 1, \ldots, m$, visit no edge of $G_a$ more than twice.
    \item For $i = 1, \ldots, m - 1$, the first letter of $w_{i + 1}$ belongs to $\cup_{j = 1}^{i} \supp(w_j)$.
\end{enumerate}
We say that $a$ is an \emph{FK word} if $m = 1$.
\end{definition}

By definition, any word admitting an interpretation as a walk in a forest visiting no edge of the forest more than twice is automatically an FK word. The constituent words of an FK sentence are FK words. If an FK sentence is at least two words long, then the result of dropping the last word is again an FK sentence. If the last word of an FK sentence is at least two letters long, then the result of dropping the last letter of the last word is again an FK sentence.

\begin{definition}[The stem of an FK sentence]\label{def:stem_FK_sentence}
Given an FK sentence $a = [w_{i}]_{i = 1}^{m}$, we define $G_{a}^{1} = (V_{a}^{1}, E_{a}^{1})$ to be the subgraph of $G_a = (V_a, E_a)$ with $V^{1}_{a} = V_a$ and $E_{a}^{1}$ equal to the set of edges $e \in E_{a}$ such that the words/walks $w_i$, $i = 1,\ldots, m$, jointly visit $e$ exactly once.
\end{definition}

The following lemma characterises the exact structure of an FK word.

\begin{lemma}[Lemma 2.1.24 in \cite{anderson2010introduction}]\label{lem:FKembed}
Suppose $w$ is an FK word. Then there is exactly one way to write $w = w_1 \cdots w_r$, where each $w_i$ is a Wigner word and they are pairwise disjoint. 
\end{lemma}

In the setting of Lemma \ref{lem:FKembed}, let $s_{i}$ be the first letter of the word $w_i$. We declare the word $s_{1} \cdots s_{r}$ to be the \emph{acronym} of the word $w$.

\textbf{FK syllabification.}
Our interest in FK sentences is mainly due to the fact that any word $w$ can be parsed into an FK sentence sequentially. In particular, one declares a new word at each time when not doing so would prevent the sentence formed up to that point from being an FK sentence. Formally, we define the FK sentence $w'$ corresponding to any given word $w$ in the following way. Suppose that $w = s_1 \cdots s_m$. We declare any edge $e \in E_{w}$ to be \emph{new} if $e = \{s_{i}, s_{i + 1}\}$ and $s_{i + 1} \notin \{s_1, \ldots, s_{i}\}$; otherwise, we declare $e$ to be \emph{old}. We now construct the FK sentence $w'$ corresponding to the word $w$ by breaking the word at each position of an old edge and the third and all subsequent positions of a new edge. Observe that any old edge gives rise to a cycle in $G_{w}$. As a consequence, by breaking the word at the old edge we remove all the cycles in $G_w$. On the other hand, all new edges are traversed at most twice as we break at their third and all subsequent occurrences. It is easy to see that the graph $G_{w'}$ remains connected since we are not deleting the first occurrence of a new edge. As a consequence, the graph $G_{w'}$ is a tree where every edge is traversed at most twice. Furthermore, by the definition of old and new edges, the first letter in the second and any subsequent word in $w'$ belongs to the support of all the previous ones. Therefore, the resulting sentence $w'$ is an FK sentence. Note that this \emph{FK syllabification} preserves equivalence, i.e. if $w \sim x$, then the corresponding FK sentences $w'\sim x'$.

The discussion about FK syllabification shows that all words can be uniquely parsed into an FK sentence. Hence we can use the enumeration of FK sentences to enumerate words of specific structures of interest.

The following two lemmas together give an upper bound on the number of equivalence classes corresponding to closed words via the corresponding FK sentences.
\begin{lemma}[Lemma 7.7 in \cite{anderson2006clt}]\label{lem:wordbyFK}
Let $\Gamma(k, l, m)$ denote the set of equivalence classes of FK sentences $a = [w_i]_{i = 1}^{m}$ consisting of $m$ words such that $\sum_{i = 1}^{m} l(w_i) = l$ and $\wt(a) = k$. Then
\begin{equation}\label{wordbyFK}
    \#\Gamma(k, l, m) \le 2^{l - m} \binom{l - 1}{m - 1} k^{2(m - 1)}.
\end{equation}
\end{lemma}

\begin{lemma}[Lemma 7.8 in \cite{anderson2006clt}]\label{lem:mbdd}
For any FK sentence $a = [w_i]_{i = 1}^{m}$, we have 
\begin{equation}\label{mval}
    m =\#E_{a}^{1} - 2\wt(a) + 2 + \sum_{i=1}^{m} l(w_i).
\end{equation}
\end{lemma}

We will also need Wick's formula for calculating joint moments of correlated Gaussians. For $k \in \bbN$, let $\cP_2(k)$ be the set of all pair-partitions of the set $\{1,2, \ldots, k\}$. 
\begin{lemma}[Wick's formula]\label{lem:Wick's_formula}
   Let $(X_1, X_2, \ldots, X_k)$ be a centered multivariate Gaussian random vector. Then
   \[
       \bbE[X_1 X_2 \cdots X_k] = \sum_{\pi \in \cP_2(k)} \prod_{\{i,j\} \in \pi} \bbE[X_i X_j].
   \]
\end{lemma}

\section{Proofs}\label{sec:proofs}
\subsection{Proof of Theorem~\ref{thm:largest_ev}}
\begin{proof}[Proof of Theorem~\ref{thm:largest_ev}]
From Corollary 2.7 (ii) of \cite{catalano2024random} it follows that the empirical spectral distribution of $n^{-1/2}X_n$ converges weakly almost surely to semicircle law. Hence we have
\[
    \liminf_{n \to \infty} \lambda_1(n^{-1/2}X_n) \ge 2 \text{ a.s.}
\]
Let $\delta = 2 + \eta$ for some $\eta >0$ and $k \in \bbN$. For brevity, write $\lambda_{1, n} = \lambda_1(n^{-1/2} X_n)$. By Markov's inequality, we have
\begin{align*}
    \bbP( \lambda_{1,n} > \delta) &= \bbP( \lambda^{2k}_{1,n} > \delta^{2k} )\\
                                  & \le \frac{\bbE[\lambda^{2k}_{1,n} ]}{\delta^{2k}} \le \frac{\bbE\Tr[(n^{-1/2}X_n)^{2k}]}{\delta^{2k}}.
\end{align*}
We have for any $k$,
\begin{align*}
\bbE\Tr[(n^{-1/2}X_n)^{2k}] &= \frac{1}{n^k}\sum_{i_{1},i_{2},\ldots ,i_{2k}} \bbE[X_{i_{1}i_{2}}\ldots X_{i_{2k},i_{1}}] \\
&= \frac{1}{n^k} \sum_{t = 1}^{2k} \sum_{\substack{i_{1},i_{2},\ldots ,i_{2k} \\ |\{i_1, i_2 \ldots, i_{2k}\}|= t}} \bbE[X_{i_{1}i_{2}}\ldots X_{i_{2k},i_{1}}]. \\
\end{align*}
Let $(i_1, i_2, \ldots, i_{2k})$ be a particular configuration of indices in the above sum. We consider the corresponding closed word $w = i_1 i_2 \cdots i_{2k} i_{1}$ which is then parsed into an FK sentence $a = [w_i]_{i = 1}^{m}$ with $\wt(a) = t$ and total length $\sum_{i = 1}^{m} l(w_i) = 2k + 1$. There can be many FK sentences with the same weight $t$ and total length $2k + 1$. We need an estimate of $ \#\Gamma(t, 2k + 1, m)$. From Lemma~\ref{lem:wordbyFK}, we have 
\begin{align}\label{eq:upper_bound_gamma}
    \#\Gamma(t, 2k + 1, m) \le 2^{2k + 1 - m} \binom{2k}{m - 1} t^{2(m - 1)}.
\end{align}
Additionally, we need to select $t$ distinct letters from the set $\{1, 2, \ldots, n\}$, which can be done in $O(n^t)$ ways. Consider the graph $G_a = (V_a, E_a)$ associated with the sentence $a$. Let $E_1 = \# E_{a}^1$, where $E_a^1$ is as in Definition~\ref{def:stem_FK_sentence}. Then from Lemma~\ref{lem:mbdd} we have
\begin{align}\label{eq:est_m}
m = E_1 - 2t + 2 + (2k + 1).
\end{align}
Using the fact that $t \le 2k$ and the relation \eqref{eq:est_m}, the upper bound in \eqref{eq:upper_bound_gamma} reduces to 
\begin{align*}
   2^{2k + 1 - m} \binom{2k}{m - 1}t^{2(m - 1)} & \le 2^{2k} \frac{(2k)^{m - 1}}{(m - 1)!} (2k)^{2(m - 1)} \\
   & \le 2^{2k} (2k)^{3(m - 1)} =  2^{2k} (2k)^{(6k - 6t + 3E_1 + 6)} . 
\end{align*}

Let $E_{a}^{2}$ be the set of edges $e \in E_a$ such that the words/walks $w_i, i=1,\ldots, m$, jointly visit $e$ exactly twice and let $E_{a}^{3}$ be the set of edges which are traversed by the words/walks thrice or more. Define $ E_i := \# E_{a}^{i}, i = 2,3$. Then it is easy to observe that
\[
    2k \ge E_1 + 2E_2 + 3E_3 \ \ \text{and} \ \ t \le E_1 + E_2 + E_3,
\]
which together imply that
\begin{align}\label{eq:upper_bound_E3}
     k - t + \frac{E_1}{2} \ge \frac{E_3}{2}. 
\end{align}

To calculate the expectation corresponding to an FK sentence, we employ Wick's formula. This requires us to keep track of which entries in the matrix $X_n$ are paired with each other. Observe that an entry $X_{i_{j - 1}, i_j }$ in the expectation corresponds to the edge $\{i_{j - 1}, i_j\}$ in the graph $G_a$. We say that two edges $\{i_{j_1 - 1}, i_{j_1}\}$  and $\{i_{j_2 - 1}, i_{j_2}\}$ of $G_a$ ``match" with each other if there is a pair partition $\pi$ of the set $\{1,2, \ldots, 2k\}$ such that $\{j_1 - 1, j_{2} - 1 \}$ is a block of $\pi$, where $ 2 \le j_1, j_2 \le 2k + 1$ with $i_{2k + 1} = i_1$.  Matchings can happen in one of the following ways:
\begin{enumerate}
    \item [(i)] some edges of $E_{a}^{1}$ can match with some edges of $E_{a}^{2}$;
    \item [(ii)] some of the remaining edges of $E_{a}^{1}$ can match with some edges of $E_{a}^{3}$; 
    \item [(iii)] some of the remaining edges of $E_{a}^{2}$ can match with some edges of $E_{a}^{3}$;
    \item [(iv)] the remaining edges of $E_{a}^{1}$ are self-matched;
    \item [(v)] some of the remaining edges of $E_{a}^{2}$ can be self-matched and others can match with one another;
    \item [(vi)] the remaining edges of $E_{a}^{3}$ can match among themselves.
\end{enumerate}

For example, let $k = 5$ and $w = 12134321451$. For this word, the expectation looks like
\[
    \bbE[ X_{12} X_{21} X_{13} X_{34} X_{43} X_{32} X_{21} X_{14} X_{45} X_{51}] .
\]
Observe that the edge $\{1,2\}$ is traversed by the walk exactly thrice, $\{3,4\}$ is traversed twice and $\{1,3\}$, $\{2,3\}, \{1,4\}, \{4,5\}$ and $\{5,1\}$ are traversed exactly once.  One possible decomposition of this expectation is 
\[
    \bbE[X_{13} X_{34}] \bbE[X_{23} X_{12}] \bbE[X_{21} X_{43} ] \bbE[X_{21} X_{14}] \bbE[X_{45} X_{51}]. 
\]
This decomposition covers the cases (i), (ii), (iii) and (iv). On the other hand, the decomposition
\[
    \bbE[ X_{13} X_{12}] \bbE[X_{21}^2] \bbE[X_{34}^2] \bbE[X_{41} X_{45}] \bbE[X_{23} X_{51}]
\]
covers the cases (ii), (iv), (v) and (vi).

Let $\gamma_1$ many edges of $E_1$ match with $E_2$,$\gamma_2$ many edges of $E_1$ match with $E_3$ and $\gamma_3$ many edges of $E_2$ match with $E_3$. For (i) we first choose $\gamma_1$ edges from $E_1$ then $\gamma_1$ edges from $2E_2$ (accounting for direction) and match them. In this case, the expectation will contribute $\frac{1}{n^{\gamma_1(1 + \varepsilon)}}$. The total contribution from (i) is thus
\begin{align}\label{eq:est_case_i}
    \binom{E_1}{\gamma_1} \binom{2E_2}{\gamma_1} \gamma_1! \frac{1}{n^{\gamma_1(1 + \varepsilon)}}. 
\end{align}
Similarly, the contingency (ii) will contribute
\begin{align}\label{eq:est_case_ii}
    \binom{E_1 - \gamma_1}{\gamma_2} \binom{E_3}{\gamma_2} \gamma_2! \frac{1}{n^{\gamma_2(1 + \varepsilon)}}. 
\end{align} 
From (iii), we get 
\begin{align}\label{eq:est_case_iii}
    \binom{2E_2 - \gamma_1}{\gamma_3} \binom{E_3 - \gamma_2}{\gamma_3} \gamma_3! \frac{1}{n^{\gamma_3(1 + \varepsilon)}}.
\end{align}
For the case (iv), the arrangement of the remaining edges will contribute 
\begin{align}\label{eq:est_case_iv}
    \frac{(E_1 - \gamma_1 - \gamma_2)!}{\big(\frac{E_1 - \gamma_1 - \gamma_2}{2}\big)! 2^{\big(\frac{E_1 - \gamma_1 - \gamma_2}{2}\big)}} \frac{1}{n^{\big(\frac{E_1 - \gamma_1 - \gamma_2}{2}\big)(1 + \varepsilon)}}. 
\end{align}

We give an upper bound for (vi) by
\begin{align}\label{eq:est_case_vi}
 \frac{(E_3 - \gamma_2 - \gamma_3)!}{\big(\frac{E_3 - \gamma_2 - \gamma_3}{2}\big)! 2^{\big(\frac{E_3 - \gamma_2 - \gamma_3}{2}\big)}}.   
\end{align}
For (v), let $p$ edges from $(2E_2 - \gamma_1 - \gamma_3)$ match with one another and the remaining edges are self-matched. Then the total contribution is
\begin{align}\label{eq:est_case_v}
   \sum_{p = 0}^{2E_2 - \gamma_1 - \gamma_3} \binom{2E_2 - \gamma_1 - \gamma_3}{p} \frac{p!}{2^{\frac{p}{2}} (\frac{p}{2})!} \frac{1}{n^{\frac{p}{2}(1 + \varepsilon)}}.   
\end{align}
Now we give upper bounds on individuals terms. We shall use the inequality $\binom{n}{r} \le \frac{n^r}{r!}$ and the fact that $E_1, E_2, E_3 \le 2k$. For \eqref{eq:est_case_i}, we get
\begin{align}\label{eq:upper_bound_case_i}
    \binom{E_1}{\gamma_1} \binom{2E_2}{\gamma_1} \gamma_1! \frac{1}{n^{\gamma_1(1 + \varepsilon)}} \le \frac{E^{\gamma_1}_1}{\gamma_1 !} \frac{(2E_2)^{\gamma_1}}{\gamma_1 !} \gamma_1!  \frac{1}{n^{\gamma_1(1 + \varepsilon)}} \le (2k)^{2\gamma_1} \frac{1}{n^{\gamma_1(1 + \varepsilon)}}.
\end{align}
Similarly, for \eqref{eq:est_case_ii},
\begin{align}\label{eq:upper_bound_case_ii}
    \binom{E_1 - \gamma_1}{\gamma_2} \binom{E_3}{\gamma_2} \gamma_2! \frac{1}{n^{\gamma_2(1 + \varepsilon)}} \le (2k)^{2\gamma_2} \frac{1}{n^{\gamma_2(1 + \varepsilon)}}, 
\end{align}
and
\begin{align}\label{eq:upper_bound_case_iii}
    \binom{2E_2 - \gamma_1}{\gamma_3} \binom{E_3 - \gamma_2}{\gamma_3} \gamma_3! \frac{1}{n^{\gamma_3(1 + \varepsilon)}} \le (2k)^{2\gamma_3} \frac{1}{n^{\gamma_3(1 + \varepsilon)}}.
\end{align}
For controlling the terms in \eqref{eq:est_case_iv},\eqref{eq:est_case_v} and \eqref{eq:est_case_vi}, we shall use the inequality $\frac{2n!}{2^n n!} \le (2n)^n$ which holds for all $n \in \bbN$. For \eqref{eq:est_case_iv}, we have  
\begin{align}\label{eq:upper_bound_case_iv}
    \frac{(E_1 - \gamma_1 - \gamma_2)!}{\big(\frac{E_1 - \gamma_1 - \gamma_2}{2}\big)! 2^{\big(\frac{E_1 - \gamma_1 - \gamma_2}{2}\big)}} \frac{1}{n^{\big(\frac{E_1 - \gamma_1 - \gamma_2}{2}\big)(1 + \varepsilon)}} \le  \frac{(E_1 - \gamma_1 - \gamma_2)^{\frac{E_1 - \gamma_1 - \gamma_2}{2}}}{n^{\big(\frac{E_1 - \gamma_1 - \gamma_2}{2}\big)(1 + \varepsilon)}} \le  \frac{(2k)^{\frac{E_1 - \gamma_1 - \gamma_2}{2}}}{n^{\big(\frac{E_1 - \gamma_1 - \gamma_2}{2}\big)(1 + \varepsilon)}}.
\end{align}
For \eqref{eq:est_case_vi}, we have the upper bound 
\begin{align}\label{eq:upper_bound_case_vi}
    \frac{(E_3 - \gamma_2 - \gamma_3)!}{\big(\frac{E_3 - \gamma_2 - \gamma_3}{2}\big)! 2^{\big(\frac{E_3 - \gamma_2 - \gamma_3}{2}\big)}} \le (2k)^{\frac{E_3 - \gamma_2 - \gamma_3}{2}}.
\end{align}
Finally, for \eqref{eq:est_case_v}, we have
\begin{align}\label{eq:upper_bound_case_v} \nonumber
    \sum_{p = 0}^{2E_2 - \gamma_1 - \gamma_3} \binom{2E_2 - \gamma_1 - \gamma_3}{p} \frac{p!}{2^{\frac{p}{2}} (\frac{p}{2})!} \frac{1}{n^{\frac{p}{2}}(1 + \varepsilon)} &\le \sum_{p = 0}^{\infty} \frac{(2E_2 - \gamma_1 - \gamma_3)^p}{p!} \frac{p!}{2^{\frac{p}{2}} (\frac{p}{2})!} \frac{1}{n^{\frac{p}{2}(1 + \varepsilon)}} \\
    & \le \sum_{p = 0}^{\infty} \frac{(2k)^p}{n^{\frac{p}{2}(1 + \varepsilon)}} =  \frac{1}{\bigg(1 - \frac{2k}{n^{\frac{1 + \varepsilon}{2}}}\bigg)}.
\end{align}
Combining the estimates \eqref{eq:upper_bound_case_i}, \eqref{eq:upper_bound_case_ii}, \eqref{eq:upper_bound_case_iii}, \eqref{eq:upper_bound_case_iv}, \eqref{eq:upper_bound_case_vi}, \eqref{eq:upper_bound_case_v}, \eqref{eq:upper_bound_E3} and using the fact that $\gamma_1 \le E_1$, we have
\begin{align*}
    &\bbE\Tr[(n^{-1/2}X_{n})^{2k}] \\
    &\preceq \frac{1}{n^k} \sum_{t = 1}^{2k} n^t 2^{2k} (2k)^{(6k - 6t + 3E_1 + 6)}(2k)^{2(\gamma_1 + \gamma_2 +\gamma_3) + \frac{E_1 - \gamma_1 - \gamma_2}{2} + \frac{E_3 - \gamma_2 - \gamma_3}{2} } \\
    &\hspace{10em} \times \frac{1}{n^{\big(\frac{E_1 + \gamma_1 + \gamma_2}{2} + \gamma_3\big)(1 + \varepsilon)}} \frac{1}{\left(1 - \frac{2k}{n^{\frac{1 + \varepsilon}{2}}}\right)} \\
    &\le  \sum_{t = 1}^{2k} \frac{1}{n^{k - t}} 2^{2k} (2k)^{(6k - 6t + 3E_1 + 6)} (2k)^{\frac{E_1 + E_3 + 3\gamma_1 + 3\gamma_3}{2} + \gamma_2} \frac{1}{n^{\big(\frac{E_1 + \gamma_1 + \gamma_2}{2} + \gamma_3\big)(1 + \varepsilon)}} \frac{1}{\left(1 - \frac{2k}{n^{\frac{1 + \varepsilon}{2}}}\right)} \\
    &\le \sum_{t = 1}^{2k} \frac{ 2^{2k} (2k)^6 }{n^{k - t}} (2k)^{6(k - t + \frac{E_1}{2}) + \frac{E_3}{2}} (2k)^{E_1 + \gamma_1 + \gamma_2 + 2 \gamma_3} \frac{1}{n^{\big(\frac{E_1 + \gamma_1 + \gamma_2}{2} + \gamma_3\big)(1 + \varepsilon)}} \frac{1}{\left(1 - \frac{2k}{n^{\frac{1 + \varepsilon}{2}}}\right)} \\
    &\le \sum_{t = 1}^{2k} \bigg(\frac{(2k)^2}{n^{\varepsilon}}\bigg)^{\frac{E_1 + \gamma_1 + \gamma_2}{2} + \gamma_3} (2k)^{ 7(k - t + \frac{E_1}{2})} \frac{1}{n^{ k - t + \frac{E_1}{2}}} \frac{2^{2k}(2k)^6}{n^{\frac{\gamma_1 + \gamma_2}{2} + \gamma_3}} \frac{1}{\left(1 - \frac{2k}{n^{\frac{1 + \varepsilon}{2}}}\right)} \\
    &\le \sum_{t = 1}^{2k} \bigg(\frac{(2k)^2}{n^{\varepsilon}}\bigg)^{\frac{E_1 + \gamma_1 + \gamma_2}{2} + \gamma_3} \bigg(\frac{(2k)^7}{n}\bigg)^{k - t + \frac{E_1}{2}} \frac{2^{2k}(2k)^6}{n^{\frac{\gamma_1 + \gamma_2}{2} + \gamma_3}} \frac{1}{\left(1 - \frac{2k}{n^{\frac{1 + \varepsilon}{2}}}\right)} \\
    &\le \bigg(\frac{(2k)^2}{n^{\varepsilon}}\bigg)^{\frac{E_1}{2}} \frac{1}{\left(1 - \frac{2k}{n^{\frac{1 + \varepsilon}{2}}}\right)} \frac{2^{2k}(2k)^6}{n^{\frac{\gamma_1 + \gamma_2}{2} + \gamma_3}} \sum_{t = 1}^{2k}  \bigg(\frac{(2k)^7}{n}\bigg)^{k - t + \frac{E_1}{2}}.
\end{align*}
As a consequence, for any $\delta = 2+ \eta$ with $\eta>0$, we see that $\frac{\bbE\Tr[(n^{-1/2}X_{n})^{2k}]}{\delta^{2k}}$ is summable over $n$ for $k \asymp (\log n)^2$. The proof is now completed using the Borel-Cantelli lemma.
\end{proof}

\subsection{Proofs of Propositions~\ref{prop:univ} and \ref{prop:rad-construction}}
\begin{proof}[Proof of Proposition~\ref{prop:univ}]
Let $d_H(A, B)$ denote the Hausdorff distance between two subsets $A, B \subset \bbR$. For a symmetric matrix $A$, let $\spec(A)$ denote its spectrum. Consider matrices of the form
\[
    Z = Z_0 + \sum_{i = 1}^N Z_i,
\]
where $Z_0$ is a $ n \times n$ deterministic matrix and $Z_1, \ldots, Z_N$ be any independent $n \times n$ self-adjoint centered random matrices. Theorem~2.6 of \cite{brailovskaya2024universality} shows that if the matrices $Z_i$ are uniformly bounded in operator norm, then
\[
    \bbP(d_H(\spec(Z), \spec(G)) > C \varpi(t)) \le n e^{-t},
\]
where $G$ is a Gaussian random matrix with the same expectation and covariance structure as $X$ and 
\[
    \varpi(t) = \sigma_*(Z) t^{1/2} + R(Z)^{1/3} \sigma(Z)^{2/3} t^{2/3} + R(Z) t,
\]
with
\begin{align*}
    \sigma(Z) &:= \|\bbE[(Z - \bbE Z)^2]\|_{\op}^{1/2}, \\
    \sigma_*(Z) &:= \sup_{\|v\| = \|w\| = 1} \bbE[|\langle v, (Z - \bbE Z) w\rangle|]^{1/2}, \\
    R(Z) &:= \bigg\| \max_{1 \le i \le n} \|Z_i\|_{\op}\bigg\|_{\infty}.
\end{align*}
Let us now calculate these parameters for the ensemble $Z = n^{-1/2} X_n(\mathbf{\Psi})$. First note that
\[
    \bbE[X^2_{ij}] = \bbE \big[\sum_k X_{ik} X_{kj}\big] = O\bigg(\frac{1}{n^{\varepsilon}}\bigg), \quad \text{and} \quad \bbE[X^2_{ii}] = n(1 + o(1)).
\]
Hence
\[
    \bbE[X^2] = n(1 + o(1)) I + O\bigg(\frac{1}{n^{\varepsilon}}\bigg) (J - I) = n(1 + o(1)) I +  O\bigg(\frac{1}{n^{\varepsilon}}\bigg) J,
\]
and consequently,
\[
    \sigma(Z) = \|\bbE[(Z - \bbE Z)^2]\|_{\op}^{1/2} =  \frac{1}{\sqrt{n}} O(\sqrt{n}) = O(1).
\]
On the other hand,
\begin{align*}
    \sigma_*(Z) \le \|\Cov(Z)\|_{\op}^{1/2} &= \frac{1}{\sqrt{n}} \big\| \Cov(X_n)\big\|^{1/2}_{\op} \\
    &\le \frac{1}{\sqrt{n}} [\text{maximum row sum of $\Cov(X_n)$}]^{1/2} \\
    &\le \frac{1}{\sqrt{n}} O(\max\{1, n^{(1 - \varepsilon) / 2}\}) \\
    &= O\bigg(\frac{1}{n^{\min\{1, \varepsilon\}/2}}\bigg).
\end{align*} 
As for $R(Z)$, if $|\psi_{\ell}| \le K$ for all $1 \le \ell \le N$, then we have
\[
    R(Z) = \frac{1}{\sqrt{n}} \bigg\| \max_{1 \le \ell \le N} \|\psi_{\ell} Q_{\ell}\|_{\op} \bigg\|_{\infty} \le \frac{K}{\sqrt{n}} \max_{1 \le \ell \le N} \|Q_{\ell}\|_{\op}.
\]
Putting everything together,
\[
    \varpi(t) =  O\bigg( \frac{1}{n^{\min\{1,\varepsilon\}/2}} \bigg) t^{1/2} + \bigg(\frac{K}{\sqrt{n}} \max_{1 \le \ell \le N} \|Q_{\ell}\|_{\op} \bigg)^{1/3} t^{2/3} +  \bigg(\frac{K}{\sqrt{n}} \max_{1 \le \ell \le N} \|Q_{\ell}\|_{\op} \bigg) t.
\]
Let $\bZ$ be an $N$-dimensional vector of i.i.d. standard Gaussians. If we choose $t = 3 \log n$, then with probability at least $ 1 - \frac{1}{n^2}$,
\begin{align*}
    &d_H(\spec(n^{-1/2} X_n(\mathbf{\Psi})), \spec(n^{-1/2}X_n(\bZ))) \\
                       &\quad= O\bigg( \bigg(\frac{\log n}{n^{\min\{1, \varepsilon\}}}\bigg)^{1/2} + \bigg( \frac{K}{\sqrt{n}} \max_{1 \le \ell \le N} \|Q_{\ell}\|_{\op} (\log n)^2 \bigg)^{1/3} + \bigg( \frac{K}{\sqrt{n}} \max_{1 \le \ell \le N} \|Q_{\ell}\|_{\op} \log n \bigg) \bigg).
\end{align*}
It is clear that the upper bound is $o(1)$ if $\max_{1 \le \ell \le N} \|Q_{\ell}\|_{\op} \ll \frac{\sqrt{n}}{(\log n)^2}$, which is precisely the condition in \eqref{ass:operator_norm_Q}. We immediately reach the conclusion that
\[
    \lambda_1(n^{-1/2} X_n(\mathbf{\Psi})) - \lambda_1(n^{-1/2} X_n(\bZ)) \convas 0.
\]
By virtue of Theorem~\ref{thm:largest_ev}, we know that $\lambda_{1}(n^{-1/2}X_n(\bZ)) \convas 2$. It follows therefore that $\lambda_1(n^{-1/2} X_n(\mathbf{\Psi})) \convas 2$.
\end{proof}

\begin{proof}[Proof of Proposition~\ref{prop:rad-construction}]
For $1 \le i \le j \le n$, let
\[
    \bv^{(ij)} = \frac{1}{\sqrt{2Np}}(\eta^{(ij)}_1, \ldots, \eta^{(ij)}_N),
\]
Since $ \|\bv^{(ij)}\|^2_2 = \frac{1}{2Np} \sum_{\ell = 1}^N (\eta^{(ij)}_\ell)^2$, by Bernstein's inequality,
\begin{align*}
    \bbP( | \|\bv^{(ij)}\|^2_2 - 1 | > t ) & \le 2 \exp \bigg( - \frac{2 N^2p^2t^2}{2Np( 1 - 2p) + 2Np t/3} \bigg) \\
    &= 2\exp \bigg( - \frac{Npt^2}{ (1 - 2p) + t/3}\bigg).
\end{align*}
Choose $t = \frac{1}{(\log n)^2}$. By a union bound, under our assumptions on $p$,
\begin{align*}
    \bbP( \exists \, i, j  \text{ s.t. }  | \|\bv^{(ij)}\|^2_2 - 1 | > t ) &\le O(n^2) \cdot \exp \bigg( - \frac{Npt^2}{ (1 - 2p) + t/3}\bigg) \\
    &= O(n^2) \cdot  \exp{\bigg( - \frac{Npt^2}{ 1 + o(1)}\bigg)}.
\end{align*}
Thus as long as $N \ge \tilde{C} n (\log n)^4$ for some large enough constant $\tilde{C} > 0$, with probability at least $1 - O(n^{-2})$, the matrices $Q_{\ell}$ will satisfy \eqref{ass:Q_entries_1}.

Now, 
$\langle \bv^{(ij)}, \bv^{(i'j')} \rangle = \frac{1}{2Np} \sum_{\ell = 1}^N \eta^{(ij)}_\ell \eta^{(i'j')}_\ell $. Note that 
\[
    \eta^{(ij)}_1 \eta^{(i'j')}_1 = \begin{cases}
    1 & \text{w.p. } 2p^2, \\
    -1 & \text{w.p. } 2p^2, \\
    0 & \text{w.p. } 1 - 4p^2.
\end{cases}
\]
Thus $\bbE [\eta^{(ij)}_1 \eta^{(i'j')}_1  ] = 0$ and $\bbE [(\eta^{(ij)}_1 \eta^{i'j'}_1)^2]= 4 p^2$. By Bernstein's inequality,
\begin{align*}
    \bbP(|\langle \bv^{(ij)}, \bv^{(i'j')} \rangle| > t) & \le 2\exp\bigg( - \frac{2N^2p^2t^2}{4Np^2 + 2Npt / 3} \bigg) = 2\exp\bigg( - \frac{Npt^2}{2p + t/3} \bigg).
\end{align*} 
Choose $t = K n^{-( 1 + \varepsilon)}$. Since $p = \Omega(\frac{\log n}{n}) \gg n^{-(1 + \varepsilon)}$, by a union bound,
\begin{align*}
    \bbP(\exists \, i,  j, i', j' \text{ s.t. } |\langle \bv^{(ij)}, \bv^{(i'j')} \rangle | > K n^{-(1 + \varepsilon)}) &\le O(n^4) \cdot \exp\bigg( -  \frac{K^2 N n^{-( 2 + 2 \varepsilon)}p}{2p + K n^{-( 1 + \varepsilon)} / 3} \bigg) \\
    & = O(n^4) \cdot \exp\big( - C_K N n^{-(2 + 2 \varepsilon)}\big),
 \end{align*}
for some constant $C_K > 0$.
Thus as long as $N \ge C n^{2 + 2\varepsilon} \log n$ for some suitably large constant $C > 0$, with probability at least $1 - O(n^{-2})$, the matrices $Q_\ell$ will satisfy
\eqref{ass:Q_entries_2}. This complete the proof of part (a).

Now we prove part (b). By modifying the proof of Theorem~1.7 in \cite{basak2017invertibility} for symmetric matrices, one can show that if $np \ge C_0 \log n$, then there exist constants $c, C > 0$ such that 
\[
    \bbP (\|Q_1\| \ge C \sqrt{np}) \le \exp(- c n p). 
\]
Therefore
\begin{align*}
    \bbP\bigg(\max_{1 \le \ell \le N} \|Q_{\ell}\|_{\op} > \sqrt{np}\bigg) & \le N \bbP(\|Q_{\ell}\|_{\op} > \sqrt{np}) \le N \exp(- c np).
\end{align*}
Thus if we choose $\frac{\log n}{n} \ll p \ll \frac{1}{(\log n)^4}$ and $N \ge C_1 n^{2 + 2\varepsilon} \log n$ for some large enough constant $C_1 > 0$, then it follows that with probability at least $1 - O(n^{-2})$, we have
\[
    \max_{1 \le \ell \le N} \|Q_{\ell}\|_{\op} = o\bigg(\frac{\sqrt{n}}{(\log n)^2}\bigg).
\]
This completes the proof.
\end{proof}

\subsection{Proofs of Theorem~\ref{thm:fluc} and Corollary~\ref{cor:fluc}}
Let $\bS = Y \bone$ and $S_i = \sum_j Y_{ij}$, i.e. $\bS = ( S_1, S_2, \ldots, S_n)^\top$. We decompose
\[
    \bone = \bv + \br,
\]
where $Y\bv = \lambda_1 \bv$ and $\bv^\top \br = 0$. Then write
\[
    \bS = Y\bone = Y\bv + Y\br = \lambda_1 \bv + Y\br
\]
Note that
\[
    \bbE \bS = L \bone,
\]
where $L = n\mu = \lambda \sqrt{n}$.

The most crucial ingredient in the proof is the following observation, referred to as the \emph{von Mises iteration} in \cite{furedi1981eigenvalues}.
\begin{lemma}[von Mises iteration]\label{lem:von_Mises}
We have
\begin{equation}\label{eq:von_Mises}
    \lambda_1 = \frac{\bS^\top \bS}{\bS^\top \bone} + \frac{\lambda_1 \br^\top Y \br - \|Y\br\|^2}{\bS^\top\bone}.
\end{equation}
\end{lemma}
\begin{proof}
We have, using the orthogonality of $\bv$ and $\br$, that
\begin{align*}
    \frac{\bS^\top \bS}{\bS^\top \bone} = \frac{\|Y\bone\|^2}{\bone^\top Y \bone} = \frac{\|\lambda_1 \bv + Y\br\|^2}{(\bv + \br)^\top (\lambda_1 \bv + Y\br)} = \frac{\lambda_1^2 \|\bv\|^2 + \|Y\br\|^2}{\lambda_1\|\bv\|^2 + \br^\top Y\br}. 
\end{align*}
A simple algebraic calculation then shows that the quantity in the right hand side above equals $\lambda_1 + \frac{\|Y\br\|^2 - \lambda_1 \br^\top Y \br}{\bS^\top \bone}$. This completes the proof.
\end{proof}

We need to control the various quantities appearing in \eqref{eq:von_Mises}. This will be done via a series of Lemmas. 
Let $Z_i := S_i - L = \sum_j X_{ij}$. Then $\bbE[Z_i] = 0$ and the following estimates hold.  
\begin{lemma}\label{lem:Zi}
    We have
    \begin{enumerate}
        \item [(i)] $\Var(Z_i) = n + O(n^{1 - \varepsilon})$.
        \item [(ii)] $\Cov(Z_i, Z_{i'}) = 1 + O(n^{1 - \varepsilon})$, $i \ne i'$.
        \item [(iii)] $\Var(Z_i^2) = O(n^2)$.
        \item [(iv)] $\Cov(Z_i^2, Z_{i'}^2) = 2 + O(n^{2 - 2 \varepsilon})$, $i \ne i'$.
    \end{enumerate}
\end{lemma}

The proof of Lemma~\ref{lem:Zi} uses Wick's formula and is given in the appendix.

\begin{lemma}\label{lem:S_dot_1}
We have
\begin{enumerate}
    \item [(i)] $\bbE[\bS^\top \bone] = \lambda n \sqrt{n}$.
    \item [(ii)] $\Var(\bS^\top \bone) = 2n^2 + O(n^{3 - \varepsilon})$.
\end{enumerate}
\end{lemma}
\begin{proof}
For (ii), we use Lemma~\ref{lem:Zi} to get
\[
    \Var(\bS^\top \bone) = \Var\big(\sum_i Z_i\big) = \sum_i \Var(Z_i) + \sum_{i \ne i'} \Cov(Z_i, Z_{i'}) = 2 n^2 + O(n^{3 - \varepsilon}).
\] 
This proves the desired result.
\end{proof}

\begin{lemma}\label{lem:S-minus-L-squared}
We have
\begin{enumerate}
    \item [(i)] $\bbE \|\bS - L\bone\|^2 = n^2 + O(n^{2 - \varepsilon})$.
    \item [(ii)] $\Var(\|\bS - L\bone\|^2) = O(n^{\max\{3, 4 - 2\varepsilon\}})$.
\end{enumerate}
    
\end{lemma}
\begin{proof}
We will use the estimates obtained in Lemma~\ref{lem:Zi}. First note that
\[
    \bbE \|\bS - L\bone\|^2 =  \sum_i \bbE [Z_i^2] = \sum_i \Var(Z_i) = n^2 + O(n^{2 - \varepsilon}).
\]
This proves (i).
On the other hand,
\begin{align*}
    \Var(\|\bS - L\bone\|^2) &= \Var\big( \sum_i Z^2_i\big) \\
    &= \sum_i \Var(Z^2_i) + \sum_{i \neq i'} \Cov( Z^2_i , Z^2_{i'}) \\
    &= O(n^3) + O(n^{4 - 2\varepsilon}) \\
    &= O(n^{\max\{3, 4 - 2\varepsilon\}}).
\end{align*}
This completes the proof of (ii).
\end{proof}
If an event occurs with probability at least $1 - O(n^{-c})$, we say that it happens with polynomially high probability (abbrv. w.p.h.p.). 
\begin{lemma}\label{lem:upper_bound_r_and_Yr}
   Suppose that $\lambda > 4$. Then there are constant $C, C' > 0$ such that
   \begin{enumerate}
       \item [(i)] $\|\br\|^2 \le \frac{C n}{\lambda^2}$ w.p.h.p.
       \item [(ii)] $\|Y\br\|^2 \le \frac{C' n^2}{\lambda^2}$ w.p.h.p.
   \end{enumerate}
\end{lemma}

\begin{proof}
Since
\[
    \bS - L\bone = (\lambda_1 -L) \bv + (Y\br - L\br),
\]
we have by Pythagoras' theorem,
\begin{equation}\label{eq:pythoagoras}
    \|\bS - L\bone\|^2 = (\lambda_1 - L)^2 \|\bv\|^2 + \|Y\br - L\br\|^2.
\end{equation}
Now by the Courant-Fischer minimax theorem and a quantitative version of Theorem~\ref{thm:largest_ev}, we have
\[
    \lambda_2(Y) \le \lambda_1(Y - \mu \bone\bone^\top) = \lambda_1(X) \le (2 + \eta)\sqrt{n}
\]
w.p.h.p. Therefore
\begin{equation}\label{eq:Yr}
    \|Y\br\| \le \lambda_2(Y) \|\br\| \le (2 + \eta) \sqrt{n} \|\br\| 
\end{equation}
w.p.h.p. It follows that
\[
    \|Y\br - L\br\| \ge |\|Y\br\| - L\|\br\|| \ge (L - \lambda_2(Y))\|\br\| \ge (\lambda - (2 + \eta)) \sqrt{n} \|\br\|
\]
w.p.h.p. It follows now from the decomposition \eqref{eq:pythoagoras} and  Lemma~\ref{lem:S-minus-L-squared} that 
\[
    \|\br\|^2 \le \frac{\|\bS - L\bone\|^2}{(\lambda - (2 + \eta))^2 n} \le \frac{C n}{\lambda^2}
\]
w.p.h.p. This proves part (i). Part (ii) then follows from \eqref{eq:Yr} and part (i).
\end{proof}

\begin{lemma}\label{lem:upper_bound_lambda1}
    We have $\lambda_1 = \lambda \sqrt{n} + O_P(\sqrt{n})$.
\end{lemma}
\begin{proof}
By Weyl's inequality and Theorem~\ref{thm:largest_ev},
\[
    \lambda_1 \le \|\mu \bone\bone^\top\|_{\op} +  \lambda_1(X) = \lambda\sqrt{n} + O_P(\sqrt{n}).  
\]
This completes the proof.
\end{proof}

\begin{lemma}\label{lem:rel_to_avg}
We have
\[
    \frac{\bS^\top \bS}{\bS^\top \bone} - \frac{\bS^\top \bone}{n} = \frac{\sqrt{n}}{\lambda} \bigg(1 + O_P\bigg(\max\bigg\{\frac{n^{-\varepsilon/2}}{\lambda}, n^{-\min\{1/2, \varepsilon\}}\bigg\}\bigg)\bigg) = \frac{\sqrt{n}}{\lambda} \bigg(1 +O_P\bigg(n^{-\frac{\min\{1, \varepsilon\}}{2}}\bigg)\bigg).
\]
\end{lemma}
\begin{proof}
We have
\begin{align*}
    \frac{\bS^\top \bS}{\bS^\top \bone} - \frac{\bS^\top \bone}{n} &= \frac{\frac{1}{n}\sum_i (S_i - L)^2 -(\frac{1}{n}\sum_i S_i - L)^2}{\sum_i S_i / n}  = \frac{\frac{1}{n}\|\bS - L\bone\|^2 - (\frac{\bS^\top \bone}{n} - L)^2}{\frac{\bS^\top \bone}{n}}.
\end{align*}
By Lemma~\ref{lem:S-minus-L-squared}, we have
\[
    \frac{1}{n}\|\bS - L\bone\|^2 = n + O_P(n^{\max\{1/2, 1 - \varepsilon\}}).
\]
Also, by Lemma~\ref{lem:S_dot_1},
\[
    \frac{\bS^\top \bone}{n} = \lambda \sqrt{n} + O_P(n^{1/2 - \varepsilon/2})
\]
and
\[
    \bbE\bigg(\frac{\bS^\top \bone}{n} - L\bigg)^2 = \frac{1}{n^2} \Var(\bS^\top \bone) = O(\max\{1, n^{1 - \varepsilon}\}). 
\]
Therefore
\begin{align*}
    \frac{\bS^\top \bS}{\bS^\top \bone} - \frac{\bS^\top \bone}{n} &= \frac{n + O_P(n^{\max\{1/2, 1 - \varepsilon\}}) + O_P(\max\{1, n^{1 - \varepsilon}\})}{\lambda \sqrt{n} + O_P(n^{1/2 - \varepsilon/2})} \\
    &= \frac{n (1 + O_P(n^{\max\{-1/2, -\varepsilon\}}))}{\lambda \sqrt{n} (1 + O_P(n^{-\varepsilon/2}/\lambda)} \\
    &= \frac{\sqrt{n}}{\lambda} (1 + O_P(\max\{n^{-\varepsilon/2}/\lambda, n^{-\min\{1/2, \varepsilon\}}\})).
\end{align*}
This completes the proof.
\end{proof}

Notice that using Lemma~\ref{lem:upper_bound_r_and_Yr}, we have the following a priori bound on $\br^\top Y \br$:
\begin{equation}\label{eq:quad_form_apriori}
    |\br^\top Y \br| \le \| \br \| \|Y\br\| \le c_1 \frac{n \sqrt{n}}{\lambda^2}
\end{equation}
w.p.h.p.

We are finally ready to prove Theorem~\ref{thm:fluc}.

\begin{proof}[Proof of Theorem~\ref{thm:fluc}]
Using Lemmas~\ref{lem:von_Mises}, ~\ref{lem:S_dot_1}, ~\ref{lem:upper_bound_r_and_Yr}, ~\ref{lem:upper_bound_lambda1}, and the estimate \eqref{eq:quad_form_apriori}, we see that
\[
     \bigg|\lambda_1 - \frac{\bS^\top \bS}{\bS^\top \bone}\bigg|
    = \frac{|\|Y\br\|^2 - \lambda_1 \br Y \br|}{|\bS^\top \bone|} = \frac{O_P(\frac{n^2}{\lambda^2}) + O_p(\frac{n^2}{\lambda})}{\lambda n \sqrt n (1 + o_p(1))} = O_P\bigg(\frac{\sqrt{n}}{\lambda^2}\bigg).
\]
This and Lemma~\ref{lem:rel_to_avg} imply that
\begin{align*}
    \lambda_1 &= \frac{\bS^\top \bS}{\bS^\top \bone} + O_P\bigg(\frac{\sqrt{n}}{\lambda^2}\bigg) \\
    &= \frac{\bS^\top \bone}{n} +\frac{\sqrt{n}}{\lambda} \bigg(1 + O_P\bigg(n^{-\frac{\min\{1, \varepsilon\}}{2}}\bigg)\bigg) + O_P\bigg(\frac{\sqrt{n}}{\lambda^2}\bigg) \\
    &= \lambda \sqrt{n} + \frac{1}{n}\sum_{i, j} X_{ij} + \frac{\sqrt{n}}{\lambda} + \frac{\sqrt{n}}{\lambda} \cdot O_P\bigg(n^{-\frac{\min\{1, \varepsilon\}}{2}}\bigg) + O_P\bigg(\frac{\sqrt{n}}{\lambda^2}\bigg).
\end{align*}
Hence
\[
    \lambda_1 - \lambda \sqrt{n} - \frac{\sqrt{n}}{\lambda} = \frac{1}{n}\sum_{i, j} X_{ij} + \frac{\sqrt{n}}{\lambda} \cdot O_P\bigg(n^{-\frac{\min\{1, \varepsilon\}}{2}}\bigg) + O_P\bigg(\frac{\sqrt{n}}{\lambda^2}\bigg).
\]
In other words,
\[
    \sqrt{n} \bigg[\lambda_1(n^{-1/2} Y_n) - \bigg(\lambda + \frac{1}{\lambda}\bigg)\bigg] = \frac{1}{n}\sum_{i, j} X_{ij} + \frac{\sqrt{n}}{\lambda} \cdot O_P\bigg(n^{-\frac{\min\{1, \varepsilon\}}{2}}\bigg) + O_P\bigg(\frac{\sqrt{n}}{\lambda^2}\bigg).
\]
This completes the proof.
\end{proof}

\begin{proof}[Proof of Corollary~\ref{cor:fluc}]
Notice that under our assumptions,
\begin{align} \nonumber
    \Var\bigg[\frac{1}{n}\sum_{i, j} X_{ij}\bigg]
    &= \frac{1}{n^2} \bigg[\sum_i \Var(X_{ii}) + 4 \sum_{i < j} \Var(X_{i, j}) \\ \nonumber
    &\qquad\qquad + 2 \sum_i\sum_{k < l} \Cov(X_{ii}, X_{kl}) + 4 \sum_{i < j}\sum_{k < l} \Cov(X_{ij}, X_{kl})\bigg] \\
    &=\max\{2 + O(1/n), O(n^{1 - \varepsilon})\} . \label{eq:sum_Xij_var}
\end{align}
Thus only for $\varepsilon \ge 1$, $\frac{1}{n}\sum_{i, j} X_{ij}$ is tight. Since the $X_{ij}$'s are jointly Gaussian, it follows that $\frac{1}{n}\sum_{i, j} X_{ij}$ has for $\varepsilon \ge 1$ an asymptotic Gaussian distribution with variance $2$. Hence, if $\lambda \gg n^{1/4}$, we have
\[
    \sqrt{n} \bigg[\lambda_1(n^{-1/2} Y_n) - \bigg(\lambda + \frac{1}{\lambda}\bigg)\bigg] \xrightarrow{d} \sqrt{2}Z,
\]
where $Z$ is a standard Gaussian. This proves Part (a).

For Part(b), we scale by $n^{\frac{1 - \varepsilon}{2}}$ to get
\begin{align*}
    n^{\varepsilon/2} \bigg[\lambda_1(n^{-1/2} Y_n) - \bigg(\lambda + \frac{1}{\lambda}\bigg)\bigg] &= \frac{1}{n^{\frac{1 - \varepsilon}{2}}} \cdot \frac{1}{n} \sum_{i, j} X_{ij} + \frac{n^{\varepsilon / 2}}{\lambda} \cdot O_P(n^{-\frac{\min\{\varepsilon, 1\}}{2}})) + O_P\bigg(\frac{n^{\varepsilon/2}}{\lambda^2}\bigg) \\
    &= \frac{1}{n^{\frac{1 - \varepsilon}{2}}} \cdot \frac{1}{n}\sum_{i, j} X_{ij} + O_P(\lambda^{-1}) + O_P\bigg(\frac{n^{\varepsilon/2}}{\lambda^2}\bigg).
\end{align*}
Since the $X_{ij}$'s are jointly Gaussian, it follows from our assumption that $\frac{1}{n^{\frac{1 - \varepsilon}{2}}} \cdot \frac{1}{n}\sum_{i, j} X_{ij}$ has an asymptotic Gaussian distribution with variance $\sigma^2$. Thus if $\lambda \gg n^{\varepsilon/4}$, we obtain that
\[
    n^{\varepsilon/2} \bigg[\lambda_1(n^{-1/2} Y_n) - \bigg(\lambda + \frac{1}{\lambda}\bigg)\bigg] \xrightarrow{d} \sigma Z.
\]
This completes the proof.
\end{proof}

\section*{Acknowledgements}
SSM was partially supported by the INSPIRE research grant DST/INSPIRE/04/2018/002193 from the Dept.~of Science and Technology, Govt.~of India, and a Start-Up Grant from Indian Statistical Institute.

\bibliographystyle{alpha}
\bibliography{refs}

\appendix
\section{Miscellaneous proofs}\label{sec:misc-proofs}

\begin{proof}[Proof of Lemma~\ref{lem:Zi}]
Parts (i) and (ii) are straightforward to prove. Indeed,
\[
    \Var(Z_i) = \sum_{j} \Var(X_{ij}) + \sum_{j \ne j} \Cov(X_{ij}, X_{ij'}) = n + O\bigg(\frac{n^2}{n^{1 + \varepsilon}}\bigg) = n + O(n^{1 - \varepsilon}),
\]
and
\[
    \Cov(Z_i, Z_{i'}) = \Var(X_{ii'}) + \sum_{j \ne i' \text{ or } j' \ne i} \Cov(X_{ij}, X_{ij'}) = n + O\bigg(\frac{n^2}{n^{1 + \varepsilon}}\bigg) = 1 + O(n^{1 - \varepsilon}).
\]
To prove (iii), we first decompose the variance as follows:
\begin{align}\label{eq:var_Z2_decomp} \nonumber
    \Var(Z^2_i) &= \Var\bigg( \sum_j X^2_{ij} + \sum_{j \neq j'} X_{ij} X_{ij'}\bigg) \\ \nonumber
    &= \Var\bigg(\sum_j X^2_{ij}\bigg) + \Var\bigg(\sum_{j \neq j'} X_{ij} X_{ij'}\bigg) + 2 \sum_{j, k \neq k'} \Cov(X^2_{ij}, X_{ik}X_{ik'}) \\ \nonumber
    &= \sum_j \Var(X^2_{ij}) + \sum_{j \neq j'} \Var(X_{ij} X_{ij'}) + \sum_{j \ne j'} \Cov(X_{ij}^2, X_{ij'}^2) \\ 
    &\qquad\qquad + \sum_{\substack{j \neq j',\, k \ne k'\\ \{j, j'\} \ne \{k, k'\}}} \Cov(X_{ij} X_{ij'}, X_{ik} X_{ik'}) + 2 \sum_{j, k \neq k'} \Cov( X^2_{ij}, X_{ik}X_{ik'}).
\end{align}
Using Wick's formula, we have
\begin{align}\label{eq:var_Z2_est_1} \nonumber
    \Var(X_{ij} X_{ij'}) &= \bbE[X^2_{ij}X^2_{ij'}] - (\bbE[X_{ij}X_{ij'}])^2 \\ \nonumber
    &= \bbE[X^2_{ij}]\bbE[X^2_{ij'}] + 2\bbE[X_{ij}X_{ij'}] \bbE[X_{ij}X_{ij'}] - (\bbE[X_{ij}X_{ij'}])^2 \\ \nonumber
    &= \bbE[X^2_{ij}]\bbE[X^2_{ij'}] + (\bbE[X_{ij}X_{ij'}])^2 \\
    &= 1 + O(n^{-(2 + 2 \varepsilon)}).
\end{align}
Similarly,
\begin{align}\label{eq:var_Z2_est_2} \nonumber
    \Cov(X_{ij}^2, X_{ij'}^2) &= \bbE[X^2_{ij}X^2_{ij'}] - \bbE[X_{ij}^2] \bbE[X_{ij'}^2] \\ \nonumber
    &= \bbE[X^2_{ij}]\bbE[X^2_{ij'}] + 2\bbE[X_{ij}X_{ij'}] \bbE[X_{ij}X_{ij'}] - \bbE[X_{ij}^2] \bbE[X_{ij'}^2] \\ \nonumber
    &= 2\bbE[X_{ij}X_{ij'}] \bbE[X_{ij}X_{ij'}] \\
    &= O(n^{-(2 + 2 \varepsilon)}).
\end{align}

\begin{align}\label{eq:var_Z2_est_3} \nonumber
    \Cov(X_{ij}X_{ij'}, X_{ik}X_{jk'}) &= \bbE[X_{ij} X_{ij'} X_{ik}X_{jk'}] - \bbE[X_{ij}X_{ij'}] \bbE[X_{ik}X_{jk'}] \\ \nonumber
    &= \bbE[X_{ij}X_{ij'}] \bbE[X_{ik}X_{ik'}] + \bbE[X_{ij}X_{ik}] \bbE[X_{ij'}X_{ik'}] \\ \nonumber 
    &\qquad\qquad + \bbE[X_{ij}X_{ik'}]\bbE[X_{ij'}X_{ik}] - \bbE[X_{ij}X_{ij'}] \bbE[X_{ik}X_{jk'}] \\ \nonumber
    &=  \bbE[X_{ij}X_{ik}] \bbE[X_{ij'}X_{ik'}] + \bbE[X_{ij}X_{ik'}]\bbE[X_{ij'}X_{ik}] \\
    &= \begin{cases}
        O(n^{-(1 + \varepsilon)}) & \text{ if } |\{j, j', k, k'\}| = 3, \\
        O(n^{-(2 + 2\varepsilon)}) & \text{ if } |\{j, j', k, k'\}| = 4. \\
    \end{cases}
\end{align}
\begin{align}\label{eq:var_Z2_est_4} \nonumber
    \Cov(X^2_{ij}, X_{ik}X_{ik'}) &= \bbE[X^2_{ij} X_{ik}X_{ik'}] - \bbE[X^2_{ij}] \bbE[X_{ik}X_{ik'}] \\ \nonumber
    &= \bbE[X^2_{ij}] \bbE[X_{ik}X_{ik'}] + 2\bbE[X_{ij}X_{ik}]\bbE[X_{ij}X_{ik'}] -\bbE[X^2_{ij}] \bbE[X_{ik}X_{ik'}] \\ \nonumber
    &= 2 \bbE[X_{ij}X_{ik}]\bbE[X_{ij}X_{ik'}] \\
    &= \begin{cases}
        O(n^{-(1 + \varepsilon)}) & \text{if } j = k \text{ or } j = k', \\
        O(n^{-(2 + 2 \varepsilon)}) & \text{ otherwise}.
    \end{cases}
\end{align}
Let $\sigma_4 = \bbE[X^4_{ij}]$. Plugging the estimates \eqref{eq:var_Z2_est_1}, \eqref{eq:var_Z2_est_2}, \eqref{eq:var_Z2_est_3} and \eqref{eq:var_Z2_est_4} into \eqref{eq:var_Z2_decomp}, we get
\begin{align*}
    \Var(Z^2_i) &= (\sigma_4 - 1) n + O(n^2) \cdot (1 + O(n^{-(2 + 2 \varepsilon)})) + O(n^2) \cdot O(n^{-(2 + 2 \varepsilon)}) \\
    &\qquad\qquad + [O(n^3) \cdot O(n^{-(1 + \varepsilon}) + O(n^4) \cdot O(n^{-(2 + 2\varepsilon})]  \\
    &\qquad\qquad\qquad\qquad + [O(n^2) \cdot O(n^{-(1 + \varepsilon}) + O(n^3) \cdot O(n^{-(2 + 2 \varepsilon)})]\\
    &= O(n^2).
\end{align*}
This proves (iii).

Now we prove (iv).
\begin{align}\label{eq:cov_z_decom} \nonumber
  \Cov( Z^2_i , Z^2_{i'}) &= \Cov \bigg( \bigg(\sum_j X_{ij}\bigg)^2 , \bigg(\sum_j X_{i'j}\bigg)^2 \bigg)\\ \nonumber
  &=  \Var(X^2_{ii'}) + \sum_{(j, j') \ne (i, i')} \Cov( X^2_{ij}, X^2_{i'j'}) + 2\sum_{j, k \neq k'} \Cov( X^2_{ij}, X_{i'k}X_{i'k'}) \\
  &\qquad +  2\sum_{j, k \neq k'} \Cov( X^2_{i'j}, X_{ik}X_{ik'}) + 4\sum_{j \neq j' , k \neq k'} \Cov( X_{ij}X_{ij'}, X_{i'k}X_{i'k'}).
\end{align}
Now
\begin{align}\label{eq:cov_est_Z2_1} \nonumber
    \Cov( X^2_{ij}, X^2_{i'j'}) &= \bbE[X^2_{ij} X^2_{i'j'}] - \bbE[X^2_{ij}]\bbE[X^2_{i'j'}] \\ \nonumber
    &= \bbE[X^2_{ij}]\bbE[X^2_{i'j'}] + 2\bbE[X_{ij}X_{i'j'}] \bbE[X_{ij}X_{i'j'}] - \bbE[X^2_{ij}]\bbE[X^2_{i'j'}] \\ \nonumber
    &=  2\bbE[X_{ij}X_{i'j'}] \bbE[X_{ij}X_{i'j'}] \\
    &= O(n^{-(2 + 2\varepsilon)}).
\end{align}
Also,
\begin{align}\label{eq:cov_est_Z2_2} \nonumber
    \Cov( X^2_{ij}, X_{i'k}X_{i'k'}) &= \bbE[X^2_{ij}X_{i'k}X_{i'k'}] - \bbE[X^2_{ij}]\bbE[X_{i'k}X_{i'k'}] \\ \nonumber
    &= \bbE[X^2_{ij}]\bbE[X_{i'k}X_{i'k'}] + 2 \bbE[X_{ij}X_{i'k}] \bbE[X_{ij}X_{i'k'}] - \bbE[X^2_{ij}]\bbE[X_{i'k}X_{i'k'}] \\ \nonumber
    &=  2 \bbE[X_{ij}X_{i'k}] \bbE[X_{ij}X_{i'k'}] \\ 
    &= O(n^{-(2 + 2\varepsilon)}).
\end{align}
Similarly,
\begin{align}\label{eq:cov_est_Z2_3} 
    \Cov( X^2_{i'j}, X_{ik}X_{ik'}) = O(n^{-(2 + 2\varepsilon)}).
\end{align}
Finally,
\begin{align}\label{eq:cov_est_Z2_4} \nonumber
    \Cov( X_{ij}X_{ij'}, X_{i'k}X_{i'k'}) &= \bbE[X_{ij}X_{ij'}X_{i'k}X_{i'k'}] - \bbE[X_{ij}X_{ij'}] \bbE[X_{i'k}X_{i'k'}] \\ \nonumber
    &= \bbE[X_{ij}X_{ij'}] \bbE[X_{i'k}X_{i'k'}] + \bbE[X_{ij}X_{i'k}] \bbE[X_{ij'}X_{i'k'}] \\ \nonumber
    &\qquad\qquad + \bbE[X_{ij}X_{i'k'}] \bbE[X_{ij'}X_{i'k}] - \bbE[X_{ij}X_{ij'}] \bbE[X_{i'k}X_{i'k'}] \\ \nonumber
    &= \bbE[X_{ij}X_{i'k}] \bbE[X_{ij'}X_{i'k'}] + \bbE[X_{ij}X_{i'k'}] \bbE[X_{ij'}X_{i'k}] \\
    &= O(n^{-(2 + 2\varepsilon)}).
\end{align}
Plugging the estimates \eqref{eq:cov_est_Z2_1}, \eqref{eq:cov_est_Z2_2}, \eqref{eq:cov_est_Z2_3} and \eqref{eq:cov_est_Z2_4} into \eqref{eq:cov_z_decom}, we get
\begin{align*}
    \Cov( Z^2_i , Z^2_{i'}) &= 2 + O(n^2) \cdot O(n^{-(2 + 2\varepsilon)}) + O(n^3) \cdot O(n^{-(2 + 2\varepsilon)}) + O(n^4) \cdot O(n^{-(2 + 2\varepsilon)}) \\
    &= 2 + O(n^{2 - 2 \varepsilon}).
\end{align*}
This completes the proof.
\end{proof}

\end{document}